\documentclass[11pt,a4paper]{amsart}
\usepackage[T1]{fontenc}
\usepackage{lmodern}
\usepackage{mathtools,amssymb}
\usepackage[margin=30mm]{geometry}
\usepackage{microtype}
\usepackage{needspace}
\usepackage[hidelinks]{hyperref}

\newtheorem{theorem}{Theorem}[section]
\newtheorem{lemma}[theorem]{Lemma}
\newtheorem{proposition}[theorem]{Proposition}
\newtheorem{corollary}[theorem]{Corollary}
\theoremstyle{definition}

\theoremstyle{remark}
\newtheorem{remark}[theorem]{Remark}

\DeclareMathOperator{\Var}{Var}
\DeclareMathOperator{\Aff}{Aff}
\DeclareMathOperator{\cont}{c}
\newcommand{\AI}{\mathsf{AI}}
\newcommand{\CAI}{\mathsf{CAI}}
\newcommand{\Avar}{\mathcal A}
\newcommand{\VV}{\mathcal V}
\newcommand{\TT}{\mathcal T}
\newcommand{\Ftwo}{\mathbb F_2}
\newcommand{\eps}{\varepsilon}
\newcommand{\poly}[1]{\mathbf{#1}}
\newcommand{\degof}[2]{m(#1,#2)}
\newcommand{\absorb}{\preceq}

\hypersetup{pdftitle={The finite basis problem for varieties generated by all semirings of a fixed order},
  pdfauthor={Aifa Wang and Lili Wang},
  pdfsubject={A structural proof for order five and a classification by order},
  pdfkeywords={semiring, finite basis, variety, cyclic identity}}

\title[Varieties generated by semirings of fixed order]
{The finite basis problem for varieties generated by all semirings of a fixed order}
\author{Aifa Wang}
\address[Aifa Wang]{School of Mathematical Sciences, Chongqing University of Technology,
Chongqing 400054, P.R. China}
\email{wangaf@cqut.edu.cn}
\author{Lili Wang}
\address[Lili Wang]{School of Mathematical Sciences, Chongqing University of Technology,
Chongqing 400054, P.R. China}
\email{wllaf@cqut.edu.cn}
\thanks{This work was supported by the National Natural Science Foundation
of China (12371024), the Science and Technology Research Program of Chongqing
Municipal Education Commission (KJZD-K2024011102) and the Chongqing Natural
Science Foundation Innovation and Development Joint Fund (Municipal Education
Commission) (CSTB2025NSCQ-LZX0067).}
\date{September 19, 2026}

\begin{document}

\begin{abstract}
Let $\Avar_n$ and $\VV_n$ be the varieties generated by all
$n$-element additively idempotent semirings (ai-semirings) and
all $n$-element semirings with
commutative addition, respectively, in the signature $(+,\cdot)$
without constants. We prove that both varieties are finitely based
if and only if $n\leq2$. The five-element case follows from cyclic
identities and a structural analysis of products of squares.
For $n\geq6$, each variety contains a finite inherently nonfinitely
based ai-semiring. We also prove $\VV_n\cap\AI=\Avar_n$ for every
$n$, where $\AI$ is the variety of all ai-semirings.
\end{abstract}

\maketitle

\section{Introduction}

Throughout this paper a \emph{semiring} is a nonempty algebra
$(S,+,\cdot)$ in which addition is associative
and commutative, multiplication is associative, and multiplication
distributes over addition on both sides. The signature contains no
constants. Neither an additive identity nor a multiplicative identity
is required. An \emph{additively idempotent semiring}, or
\emph{ai-semiring}, additionally satisfies $x+x\approx x$.
We write $\AI$ for the variety of ai-semirings and $\CAI$ for its
subvariety satisfying $xy\approx yx$.

An algebra is \emph{finitely based} if all its identities follow
from a finite set of identities; otherwise it is \emph{nonfinitely
based}. The same terminology applies to its generated variety.
For a class $\mathcal K$ of algebras, let $\Var(\mathcal K)$ denote
the variety it generates. For every positive integer $n$, put
\begin{align*}
\Avar_n&=\Var\{S:S\in\AI,\ |S|=n\},\\
\VV_n&=\Var\{S:S\text{ is a semiring},\ |S|=n\}.
\end{align*}
For order two, Shao and Ren~\cite{SR15} proved that every subvariety
of $\Avar_2$ is finitely based. Wang, Wang, Li, and Yin~\cite{WWLY26}
obtained the same conclusion for $\VV_2$. For individual
three-element ai-semirings, the results of Zhao et al.~\cite{ZRCS20}
and Jackson, Ren, and Zhao~\cite{JRZ22} give a classification in
which $S_7$ is the only nonfinitely based algebra, up to isomorphism.
The four-element problem has been studied in a series of papers
by Ren et al.~\cite{RLZC25,RLYC26} and
Yue et al.~\cite{YRZS25}. Further nonfinite basis results for
particular four-element ai-semirings appear in~\cite{YR26,YRG26}.
General sufficient conditions and the structure of limit varieties
are considered in~\cite{JRZ22,LRY26}.

The varieties $\Avar_n$ and $\VV_n$ are generated jointly by all
algebras of the prescribed order. Their finite basis problem is
distinct from the problem for the individual generators: a variety
containing a nonfinitely based subvariety need not itself be
nonfinitely based. Wang, Ju, and Wang~\cite{WJW26} proved that
$\Avar_3$ has no equational basis with a uniform bound on the number
of variables. They also obtained a nonfinite basis result for the
joint variety of three-element semirings under a convention
permitting noncommutative addition. Yue and Shao~\cite{YS26}
proved the corresponding result for $\Avar_4$.

We prove the five-element case and combine it with the known
small-order results and a criterion of Dolinka, Gusev, and
Volkov~\cite{DGV25} to obtain the following classification.

\begin{theorem}\label{thm:main}
For every positive integer $n$, each of $\Avar_n$ and $\VV_n$
is finitely based if and only if $n\leq2$.
\end{theorem}

Our sufficient condition for order five uses the isolated-class
arguments of~\cite{WJW26,YS26}. Two five-element commutative
auxiliary semirings determine an equivalence class consisting of
exactly two terms. As the cycle length increases, an identity used
to pass between the two terms must involve an unbounded number of
variables. The cyclic identities
contain an additional summand $pz^2$, which is needed for their
validity on all five-element generators. We treat substitutions
by arbitrary terms. In particular, the degree restrictions show that
a variable occurring twice in the relevant word must have a
single-word image.

Validity of the cyclic identities is proved by analysing products
of squares in semigroups of order at most five. If the
idempotents form a subsemigroup, then so do the squares. In an
ai-semiring of order at most five, a product $b$ of squares that is
not itself a square satisfies $bz=bz^2$.
For a square $b$, the possible failures of
$bz\leq b+bz^2+z$ reduce to three five-element ai-semirings.
Two satisfy the inequality after arbitrary left multiplication;
the remaining case, the flat extension of $C_4$, is treated by
an odd-cycle argument.

For $n\geq6$, the two six-element ai-semirings in~\cite{DGV25}
have an inherently nonfinitely based direct product. This product
belongs to every $\Avar_n$ with $n\geq6$. Finally, an additive
retraction gives $\VV_n\cap\AI=\Avar_n$ and transfers the
nonfinite basis results to semirings with commutative addition.

Section~\ref{sec:preliminaries} fixes the notation and the form of
equational deduction. The auxiliary semirings and cyclic identities
are introduced in Section~\ref{sec:test}. Sections
\ref{sec:isolation} and~\ref{sec:substitution} establish the
sufficient condition, and Section~\ref{sec:five} applies it to
$\Avar_5$. Section~\ref{sec:six} treats orders at least six.
The additive retraction and the proof of
Theorem~\ref{thm:main} are given in Section~\ref{sec:transfer}.

\section{Preliminaries}\label{sec:preliminaries}

We use the standard facts of universal algebra and equational logic;
see Burris and Sankappanavar~\cite{BS81}. A variety is
\emph{locally finite} if each of its finitely generated members is
finite.

\begin{lemma}\label{lem:padding}
Every ai-semiring of order at most $n$ belongs to $\Avar_n$. The
varieties $\Avar_n$ and $\VV_n$ are locally finite.
\end{lemma}
\begin{proof}
Given a semiring $S$, adjoin a new element $\omega$ absorbing for
both operations. Each associativity or distributivity instance
involving $\omega$ has value $\omega$ on both sides; the remaining
instances are unchanged. If $S$ is an ai-semiring, so is the
extension. Iterating this construction embeds every ai-semiring
of order at most $n$ into one of order $n$.

For fixed $n$ there are finitely many operation tables on an
$n$-element set. Thus each of the two joint varieties is generated
by a finite direct product of finite algebras. If $D$ is a finite
algebra, its free algebra on $r$ generators in $\Var(D)$ is
represented by the $r$-ary term functions on $D$, so is finite.
Consequently both joint varieties are locally finite.
\end{proof}

Fix a countably infinite alphabet $X$. Let $X^+$ and $X^*$ be its
free semigroup and free monoid, with empty word $\eps$. Their
commutative counterparts are denoted by $X_c^+$ and $X_c^*$.
By distributivity and additive idempotence, an ai-semiring term is
represented by a finite nonempty set of words. Addition is union,
and multiplication is setwise concatenation; see
\cite[Section~1]{RLZC25}. We write terms in
bold type and their individual words in ordinary type. In the
commutative setting we use finite nonempty subsets of $X_c^+$.
The notation $\cont(w)$ denotes the content of a word, and
$\degof{x}{w}$ its multiplicity of $x$. The content of a term is
the union of its word contents.

The additive order of an ai-semiring is
$a\leq b$ if and only if $a+b=b$. For terms we write
\[
\poly u\absorb\poly v
\quad\text{for the identity}\quad
\poly u+\poly v\approx\poly v.
\]
This differs from literal inclusion of word sets, denoted by
$\poly u\subseteq\poly v$. Both operations preserve the additive
order.

Every identity $\poly u\approx\poly v$ is equivalent, modulo the
ai-semiring identities, to all the inequalities
$u\absorb\poly v$ for $u\in\poly u$ and
$v\absorb\poly u$ for $v\in\poly v$. This replacement introduces
no variables. We may therefore work with inequalities
$r\absorb\poly s$, where $r$ is a word.

A substitution maps each variable to a nonempty term. When a
substituted term is expanded, different occurrences of the same
variable may contribute different words from its image.

\begin{lemma}\label{lem:deduction}
In the free commutative ai-semiring, an elementary application of
an inequality $r\absorb\poly s$ has the form
\begin{equation}\label{eq:deduction}
\poly l+\poly K\varphi(\poly s)
\ \longleftrightarrow\
\poly l+\poly K\varphi(\poly s)+\poly K\varphi(r).
\end{equation}
The additive context $\poly l$ and the multiplicative context
$\poly K$ may be absent. Every equational deduction is a finite
sequence of such steps.
\end{lemma}
\begin{proof}
Equational completeness describes deductions as contextual
substitution instances of defining identities. Expand a context
with one distinguished occurrence of its argument by
distributivity. Addition contributes to $\poly l$, and multiplication
contributes to the factors on either side of the occurrence. In
the commutative setting these factors combine into $\poly K$.
Applying this to $\poly s\approx\poly s+r$ gives
\eqref{eq:deduction}. For explicit formulations in ai-semirings,
see~\cite[Lemma~1.1]{RLZC25} and~\cite[Lemma~2.1]{YS26}.
\end{proof}

An absent multiplicative context is represented by $\{\eps\}$,
and an absent additive context by the empty set. These are
conventions for contexts only; neither the empty word nor the
empty set is an allowed image of a variable.

\section{Auxiliary semirings and cyclic identities}
\label{sec:test}

Let
\[
H^0=\{0,1,a,a^2,\infty\},\qquad
0<1<a<a^2<\infty,
\]
with maximum as addition. The element $0$ is a multiplicative
zero. On the other four elements put $a^0=1$, $a^3=\infty$, and
\[
a^ia^j=a^{\min(i+j,3)}\qquad(0\leq i,j\leq3).
\]
Multiplication is associative by truncated addition of exponents,
and adjoining the multiplicative zero preserves associativity.
Each multiplication map is order preserving on the chain, hence
preserves binary joins. Thus $H^0$ is a commutative ai-semiring.
Its subalgebra $D=\{0,1\}$ is the two-element distributive lattice,
with join as addition and meet as multiplication.

Let $C_4=\langle g\rangle$ and let
\[
G=C_4\cup\{\infty\}
\]
be its flat extension: $u+u=u$, $u+v=\infty$ for $u\ne v$, and
$\infty$ is absorbing for multiplication. This is a special case
of the flat construction discussed in~\cite[Section~2]{JRZ22}.
Multiplication by a
group element permutes $C_4$ and fixes $\infty$, so preserves flat
addition; multiplication by $\infty$ is constant. Hence $G$ is
also a commutative ai-semiring. Element names in these definitions
do not introduce constants into the signature.

Put
\begin{equation}\label{eq:testvariety}
\TT=\Var(H^0,G)\subseteq\Avar_5\cap\CAI.
\end{equation}
Let
\[
\mathcal N=\{k\in\mathbb N:k\geq5,\ k\text{ is odd}\}.
\]
For $k\in\mathcal N$, take distinct variables
$x_1,\ldots,x_k,y_1,\ldots,y_k,z$ and read subscripts of $y$
modulo $k$. In the noncommutative free ai-semiring define
\begin{align}
p_k&=x_1^2x_2^2\cdots x_k^2,
&q_k&=p_kz,\label{eq:noncomm-words}\\
\poly f_k&=p_k+p_kz^2+x_1^2z+
\sum_{i=1}^k y_i^2y_{i+1}^2x_i^2z.
\label{eq:noncomm-poly}
\end{align}
The multiplication order is the displayed order. Denote the
identity $q_k\absorb\poly f_k$ by $\sigma_k$.

In the commutative free ai-semiring write
\begin{gather}
P_k=\prod_{i=1}^k x_i^2,\quad
R_k=P_kz^2,\quad A_k=x_1^2z,\quad Q_k=P_kz,
\label{eq:comm-words}\\
B_i=x_i^2y_i^2y_{i+1}^2z\quad(1\leq i\leq k),
\qquad
\poly F_k=P_k+R_k+A_k+\sum_{i=1}^kB_i.
\label{eq:comm-poly}
\end{gather}
Thus $\poly F_k$ and $Q_k$ are the commutative images of
$\poly f_k$ and $q_k$. All words in $\poly F_k+Q_k$ are distinct.

\begin{lemma}\label{lem:test-validity}
The variety $\TT$ satisfies $\sigma_k$ for every $k\in\mathcal N$.
\end{lemma}
\begin{proof}
In $H^0$, if $z=0$, then $P_kz=0$. Otherwise $z\leq z^2$, so
$P_kz\leq P_kz^2$. This proves absorption in $H^0$.

Consider $G$. Unless $\poly F_k=\infty$, all its words have a
common group value $c=P_k$. Every variable then takes a group
value because it occurs in $\poly F_k$. Since $P_k$ is a product
of squares, $c^2=1$, and $R_k=P_k$ gives $z^2=1$.
Multiplying the $k$ equalities $B_i=c$ yields
\[
c^k=P_k\left(\prod_{i=1}^k y_i\right)^4z^k=P_kz.
\]
As $k$ is odd, $c^k=c=P_k$, whence cancellation gives $z=1$.
Therefore $Q_k=c=\poly F_k$. The case $\poly F_k=\infty$ is
automatic.
\end{proof}

We recall a linear-algebra observation used below. Its variables
are scalars over a field and are unrelated to the semiring
variables.

\begin{lemma}\label{lem:affine}
Let $L,L_0,\ldots,L_t$ be linear forms on a finite-dimensional
vector space over a field. If
\[
L_0(\xi)=\cdots=L_t(\xi)
\quad\Longrightarrow\quad L(\xi)=L_0(\xi),
\]
then $L\in\Aff\{L_0,\ldots,L_t\}$.
\end{lemma}
\begin{proof}
Let $M$ have rows the coefficient vectors of $L_i-L_0$ for
$1\leq i\leq t$. The hypothesis says that $L-L_0$ annihilates
$\ker M$. The annihilator of $\ker M$ is the row space of $M$.
Thus $L-L_0$ is a linear combination of the $L_i-L_0$, which is
equivalent to the asserted affine representation.
\end{proof}

\section{An equivalence class with two terms}
\label{sec:isolation}

\begin{lemma}\label{lem:absorbed}
For every $k\in\mathcal N$ and every nonempty commutative word $w$,
\begin{equation}\label{eq:absorbed}
\TT\models w\absorb\poly F_k
\quad\Longleftrightarrow\quad
w\in\poly F_k+Q_k.
\end{equation}
\end{lemma}
\begin{proof}
The reverse implication follows from Lemma~\ref{lem:test-validity}.
Suppose that $\TT\models w\absorb\poly F_k$.

Assign $a$ to one variable and $1$ to all the others in $H^0$.
For a variable in $\poly F_k$, the value of $\poly F_k$ is $a^2$;
for an outside variable, it is $1$. Hence
\begin{equation}\label{eq:degree-bound}
\cont(w)\subseteq\{x_1,\ldots,x_k,y_1,\ldots,y_k,z\},
\qquad \degof{v}{w}\leq2.
\end{equation}
Assign $g^2$ to one $x_i$ or one $y_j$ and $1$ to the other
variables in $G$. All words of $\poly F_k$ then have value $1$.
Consequently
\begin{equation}\label{eq:even-bound}
\degof{x_i}{w},\degof{y_j}{w}\in\{0,2\}.
\end{equation}

Every pair of different $x$- or $y$-variables occurring in $w$
must occur together in some word of $\poly F_k$. Otherwise,
assigning $a$ to the two variables and $1$ to all the others
in $H^0$ gives $\poly F_k\leq a^2$ but $w=\infty$, by
\eqref{eq:even-bound}. In particular,
\begin{equation}\label{eq:incidence}
x_i\text{ can occur with }y_j\text{ only if }j\in\{i,i+1\},
\end{equation}
and two different $y$-variables in $w$ must be adjacent on the
cycle.

The subsemiring $D$ implies that at least one word of $\poly F_k$
has content contained in $\cont(w)$. Indeed, assign $1$ to the
variables of $w$ and $0$ to all other variables. If no such word
existed, then $w=1$ and $\poly F_k=0$.

\smallskip\noindent
\emph{Case 1: $w$ contains every $x_i$.}
Each $y_j$ occurs together with only two $x$-variables in
$\poly F_k$. Since $k\geq5$, condition~\eqref{eq:incidence}
excludes every $y_j$ from $w$. Equations
\eqref{eq:degree-bound} and~\eqref{eq:even-bound} now give
$w\in\{P_k,Q_k,R_k\}$.

\smallskip
In the remaining cases choose an $x_j$ absent from $w$. Set
$x_j=0$, $z=a$, and every other variable equal to $1$ in $H^0$.
Then $P_k=R_k=0$ and $\poly F_k=a$. Hence
$\degof{z}{w}\leq1$. The content condition supplied by $D$
forces $w$ to contain $A_k$ or some $B_i$ in content, so
\begin{equation}\label{eq:z-once}
\degof{z}{w}=1.
\end{equation}

\smallskip\noindent
\emph{Case 2: $\cont(B_i)\subseteq\cont(w)$ for some $i$.}
The occurrence of $x_i$ excludes all $y$-variables except
$y_i,y_{i+1}$. An additional $x_j$ would have to occur together
with both $y_i$ and $y_{i+1}$, which forces $j=i$. Thus
$\cont(w)=\cont(B_i)$, and the degree restrictions give $w=B_i$.

\smallskip\noindent
\emph{Case 3: $\cont(A_k)\subseteq\cont(w)$.}
Write
\[
\degof{x_i}{w}=2\delta_i,\qquad
\degof{y_j}{w}=2\eta_j,
\qquad \delta_i,\eta_j\in\{0,1\},\quad\delta_1=1.
\]
Condition~\eqref{eq:incidence} gives $\eta_j=0$ for
$3\leq j\leq k$. Consider the valuations in $G$ given by
\[
x_i=g^{u_i},\qquad y_j=g^{v_j},\qquad z=g^{2t},
\qquad u_i,v_j,t\in\Ftwo.
\]
The values of all relevant words lie in $\{1,g^2\}$. After
dividing their exponents by $2$, their values are represented by
the following linear forms over $\Ftwo$:
\begin{align*}
P_k,R_k&:\ \sum_i u_i, & A_k&:\ u_1+t,\\
B_i&:\ u_i+v_i+v_{i+1}+t,
&w&:\ \sum_i\delta_i u_i+\sum_j\eta_jv_j+t.
\end{align*}
When the forms for $\poly F_k$ have a common value, absorption
forces the form for $w$ to have that value. By
Lemma~\ref{lem:affine}, it is an affine combination of the former
forms. Write its coefficients as
$\alpha,\gamma,\beta_1,\ldots,\beta_k$, merging the equal forms
for $P_k$ and $R_k$ into the coefficient $\alpha$.
Comparing the coefficient sum and the $t$-coordinate gives
\[
1=\alpha+\gamma+\sum_i\beta_i,
\qquad 1=\gamma+\sum_i\beta_i,
\qquad\alpha=0.
\]
The other coordinates give
\begin{equation}\label{eq:affine-coordinates}
\delta_i=\gamma\mathbf1_{i=1}+\beta_i,
\qquad \eta_j=\beta_j+\beta_{j-1}.
\end{equation}
As $\eta_j=0$ for $j\geq3$, we obtain
$\beta_2=\cdots=\beta_k$, and hence
$\delta_2=\cdots=\delta_k$. If these entries are $1$, Case~1
applies. Otherwise they are all $0$, and
\eqref{eq:affine-coordinates} gives
$\eta_1=\eta_2=\beta_1$. Thus $w=A_k$ or $w=B_1$.
The cases exhaust the content condition, proving the lemma.
\end{proof}

\begin{lemma}\label{lem:class}
For every commutative term $\poly u$ and every $k\in\mathcal N$,
\[
\TT\models\poly u\approx\poly F_k
\quad\Longleftrightarrow\quad
\poly u=\poly F_k\ \text{or}\ \poly u=\poly F_k+Q_k.
\]
\end{lemma}
\begin{proof}
If $\poly u\approx\poly F_k$ holds in $\TT$, then each word of
$\poly u$ is absorbed by $\poly F_k$. Lemma~\ref{lem:absorbed}
gives $\poly u\subseteq\poly F_k+Q_k$. We show that every word
of $\poly F_k$ belongs to $\poly u$.

For $P_k$, evaluate in $D$ with all $x_i=1$ and all $y_i=z=0$.
Only $P_k$ has value $1$. For $A_k$, evaluate in $D$ with
$x_1=z=1$ and all other variables $0$. Again only the selected
word has value $1$.

For $R_k$, evaluate in $H^0$ with $z=a$ and all other variables
$1$. Then $R_k=a^2$, while every other word of $\poly F_k+Q_k$
has value at most $a$. For $B_i$, set $y_i=y_{i+1}=a$ and all
other variables equal to $1$ in $H^0$. Only $B_i$ contains both
selected $y$-variables, so $B_i=\infty$ and all other words
have value at most $a^2$.

Each valuation separates one word of $\poly F_k$ from the sum
of all the other words of $\poly F_k+Q_k$. Hence
$\poly F_k\subseteq\poly u\subseteq\poly F_k+Q_k$.
The converse follows from Lemma~\ref{lem:test-validity}.
\end{proof}

\section{A sufficient condition for nonfinite basability}
\label{sec:substitution}

We first show that if a nonempty sum of words from $\poly F_k$
absorbs $Q_k$ in $G$, then it contains every $B_i$.

\begin{lemma}\label{lem:locals-required}
Let $\varnothing\ne\poly E\subseteq\poly F_k$. If
$G\models Q_k\absorb\poly E$, then every $B_i$ belongs to
$\poly E$.
\end{lemma}
\begin{proof}
Suppose that $B_j$ is absent. Set $u_2=1$ and $u_i=0$ for
$i\ne2$, so $u_1=0$ and $\sum_i u_i=1$ in $\Ftwo$.
Choose $v_1,\ldots,v_k\in\Ftwo$ satisfying
\begin{equation}\label{eq:path}
v_i+v_{i+1}=u_i\qquad(i\ne j).
\end{equation}
Deleting the $j$th edge of a cycle leaves a path, so these
equations are solvable by choosing one endpoint value and
proceeding along the path. Evaluate in $G$ by
\[
x_i=g^{u_i},\qquad y_i=g^{v_i},\qquad z=g^2.
\]
Then $P_k=R_k=A_k=g^2$. For $i\ne j$,
equation~\eqref{eq:path} gives $B_i=g^2$ as well. But
$Q_k=P_kz=1$. Distinct group elements are incomparable in
the flat additive order, so $Q_k\absorb\poly E$ fails.
\end{proof}

\begin{lemma}\label{lem:substitution}
Let $\TT\models r\absorb\poly s$, where $r\in X_c^+$ and
$\poly s$ is a commutative term. Suppose that a substitution
$\varphi$ and a word $h\in X_c^*$ satisfy
\begin{equation}\label{eq:instance}
h\varphi(\poly s)\subseteq\poly F_k,
\qquad Q_k\in h\varphi(r).
\end{equation}
Then $|\cont(\poly s)|\geq k$.
\end{lemma}
\begin{proof}
Every word of $h\varphi(r)$ is absorbed by $\poly F_k$ in
$\TT$, because $r\absorb\poly s$ and
$h\varphi(\poly s)\subseteq\poly F_k$. Therefore
Lemma~\ref{lem:absorbed} gives
\begin{equation}\label{eq:all-images}
h\varphi(r)\subseteq\poly F_k+Q_k.
\end{equation}
Every letter has degree at most $2$ in the words on the right.
For each variable $v\in\cont(r)$, choosing the same
nonempty image word at all its occurrences shows that
\begin{equation}\label{eq:source-degree}
\degof{v}{r}\leq2.
\end{equation}
Also, $\cont(r)\subseteq\cont(\poly s)$. Otherwise a variable
occurring only in $r$ could be assigned $\infty$ in $G$, with
all other variables of $r$ and $\poly s$ assigned $1$, contradicting
$r\absorb\poly s$.

We claim that if $\degof{v}{r}=2$, then $\varphi(v)$ consists
of a single commutative word. By~\eqref{eq:all-images}, every
word in $\varphi(v)$ is supported on the $x$-variables,
the $y$-variables, and $z$, and every letter has degree at most
$1$ in it: choosing that word at both occurrences of $v$ gives
the asserted bound. Choose an occurrence of $v$ in a word of
$\poly s$ and fix all other choices in an expansion. Replacing
the choice at this occurrence by $w,w'\in\varphi(v)$ gives
two words in $\poly F_k$ after multiplication by $h$.
Every $x$- and $y$-degree in $\poly F_k$ is even. Thus the
corresponding degrees of $w$ and $w'$ differ by an even number.
They are each $0$ or $1$, so they are equal.

If the $z$-degrees of $w$ and $w'$ differed, they would differ
by exactly $1$. The two resulting words of $\poly F_k$ would
then differ only by one occurrence of $z$. No such pair exists:
the only distinct words in $\poly F_k$ with the same
$x$- and $y$-degrees are $P_k$ and $R_k$, whose $z$-degrees
differ by $2$. Therefore the $z$-degrees also agree and
$w=w'$, proving the claim.

Take an expansion in~\eqref{eq:instance} producing $Q_k$.
Each variable occurring once in $r$ has one selected
image word, and the preceding claim ensures that a variable
occurring twice has the same word selected at both occurrences.
Choose those images and choose arbitrary image words for the
remaining variables. This defines a word substitution
$\psi$ with
\begin{equation}\label{eq:coherent}
h\psi(r)=Q_k,
\qquad h\psi(\poly s)\subseteq\poly F_k.
\end{equation}
The identity $r\absorb\poly s$ is valid in $G$, so
$G\models Q_k\absorb h\psi(\poly s)$. By
Lemma~\ref{lem:locals-required}, every $B_i$ occurs in
$h\psi(\poly s)$. Hence $h$ divides all the $B_i$ and
contains no $x$-variable.

For each $i$, choose a word $s_i\in\poly s$ such that
$h\psi(s_i)=B_i$. Some variable $v_i$ of $s_i$ has an image
$\psi(v_i)$ containing $x_i$. This image divides $B_i$.
For $j\ne i$ the word $B_j$ contains no $x_i$, so the same
variable cannot be used as $v_j$. The variables
$v_1,\ldots,v_k$ are distinct members of $\cont(\poly s)$.
Thus $|\cont(\poly s)|\geq k$.
\end{proof}

\begin{theorem}\label{thm:criterion}
Let $\mathcal W$ be an ai-semiring variety containing $H^0$
and $G$. If $\mathcal W$ satisfies $\sigma_k$ for infinitely
many $k\in\mathcal N$, then $\mathcal W$ has no equational
basis whose identities use a uniformly bounded number of
variables. In particular, $\mathcal W$ is nonfinitely based.
\end{theorem}
\begin{proof}
Suppose that $\Sigma$ is such a basis, with variable bound $m$.
Choose $k\in\mathcal N$ with $k>m$ and
$\mathcal W\models\sigma_k$. There is a deduction of
$\sigma_k$ from $\Sigma$.

Adjoin the identity $xy\approx yx$ to the deduction system
and work with commutative terms. A deduction remains
valid in this enlarged system, and all its identities hold in
$\TT$, because $\TT\subseteq\mathcal W$ and $\TT\subseteq\CAI$.
Replace each basis identity by word absorption inequalities.
The resulting set still has variable bound $m$.

By Lemma~\ref{lem:deduction}, a sequence of contextual
substitution instances takes $\poly F_k$ to $\poly F_k+Q_k$.
All intermediate terms are equivalent to $\poly F_k$ in
$\TT$. Lemma~\ref{lem:class} therefore leaves only the two
possible terms $\poly F_k$ and $\poly F_k+Q_k$.

Consider the first step changing the former into the latter.
A contraction in~\eqref{eq:deduction} introduces no word,
so this step expands an inequality $r\absorb\poly s$:
\[
\poly l+\poly K\varphi(\poly s)
\longrightarrow
\poly l+\poly K\varphi(\poly s)+\poly K\varphi(r).
\]
Choose a word $h$ of the multiplicative context that produces
$Q_k$; if the context is absent, take $h=\eps$. The left side
gives $h\varphi(\poly s)\subseteq\poly F_k$, while
$Q_k\in h\varphi(r)$. Lemma~\ref{lem:substitution} implies
$|\cont(\poly s)|\geq k>m$, a contradiction.
\end{proof}

\begin{corollary}\label{cor:product25}
The $25$-element commutative ai-semiring $H^0\times G$ is
nonfinitely based. Its generated variety has no basis with a
uniform bound on the number of variables.
\end{corollary}
\begin{proof}
Its generated variety is $\TT$. Apply
Lemma~\ref{lem:test-validity} and Theorem~\ref{thm:criterion}.
\end{proof}

\section{The five-element generators}
\label{sec:five}

For a semigroup $S$, write
\[
 E(S)=\{e\in S:e^2=e\},\qquad
 Q(S)=\{x^2:x\in S\},\qquad K(S)=\langle Q(S)\rangle.
\]
The brackets here denote multiplicative subsemigroup generation.

\subsection{Products of squares in small semigroups}

\begin{lemma}\label{lem:cyclic-codimension}
If a semigroup has a cyclic subsemigroup with at most one element
outside it, then its set of squares is a subsemigroup.
\end{lemma}
\begin{proof}
Only $S=\langle a\rangle\cup\{x\}$ needs consideration.
Every square is either an even power of $a$ or $x^2$.
Products of two even powers are squares, as is $(x^2)^2$.
For a mixed product, if $xa=x$, then $x^2a^{2i}=x^2$.
Otherwise $xa=a^j$ for some $j\geq1$, and
\[
 x^2a^{2i}=x(xa^{2i})
 =xa^{j+2i-1}=a^{2j+2i-2}
 \qquad(i\geq1).
\]
This is again a square. The argument for $a^{2i}x^2$ is the
left--right dual.
\end{proof}

\begin{lemma}\label{lem:orthodox-small}
Let $S$ be a semigroup of order at most five. If $E(S)$ is a
subsemigroup, then $Q(S)$ is a subsemigroup.
\end{lemma}
\begin{proof}
If a cyclic subsemigroup has at least four elements, apply
Lemma~\ref{lem:cyclic-codimension}. We may therefore suppose that
every cyclic subsemigroup has order at most three. If every square
is idempotent, there is nothing to prove. A nonidempotent square
must then come from one of the following two cyclic subsemigroups:
\[
 C_3,\qquad N_3=\{a,b,e\},\quad a^2=b,\quad a^3=a^4=e.
\]
This follows immediately by considering the index and period of
a cyclic semigroup of order at most three.

First suppose that $S$ contains a group $H\cong C_3$, with identity
$e$. All elements of $H$ are squares. Unless exactly one element
of $S\setminus H$ is a square, $Q(S)$ is either $H$ or $S$.
Write $S\setminus H=\{c,d\}$ in the remaining case, with
$c\in Q(S)$ and $d\notin Q(S)$; a missing second element would
already be covered by Lemma~\ref{lem:cyclic-codimension}.
If $c$ is idempotent, closure of $E(S)$ gives
$ec,ce\in\{e,c\}$. If $ec=e$, then $gc=g$ for every $g\in H$.
If $ec=c$, the group $H$ acts on $c$ by left multiplication.
Its orbit is disjoint from $H$ and has at most two elements, so
it has one element; thus $gc=c$. The same argument on the right
shows that all mixed products lie in $H\cup\{c\}$.
If $c$ is not idempotent, its square root must be $d$.
The cyclic subsemigroup of $d$ cannot be a group of order three,
since $d$ is not a square. Consequently
$d^2=c$ and $d^3=d^4=e$, giving $ec=ce=e$ and again
$gc=cg=g$. Thus $Q(S)$ is closed in this case.

It remains to consider $N_3\subseteq S$. The element $a$ is not
a square: if $a=x^2$, a cyclic subsemigroup of order at most
three containing $x$ would either make $a$ a group element or
make $a^2$ idempotent. Both are impossible in $N_3$.
If there is at most one element outside $N_3$, use
Lemma~\ref{lem:cyclic-codimension}. Otherwise write the two
remaining elements as $c,d$. If neither is a square, then
$Q(S)=\{b,e\}$ is closed.

Suppose that only $c$ is a square. If $c$ is nonidempotent, its
root is $d$, and the same cyclic-semigroup argument gives
$d^2=c$ and $d^3=d^4=e$. In this case $e$ is a zero of $S$,
and every element has a power equal to $e$.
The products $bc$ and $cb$ cannot equal $a$ or $d$.
For example, $bc=a$ gives $a=a(ac)$, and $bc=d$ gives
$d=(bd)d$; iterating the indicated multiplier gives $a=e$ or
$d=e$. The two reversed products are treated in the same way.
Hence $Q(S)=\{b,c,e\}$ is closed.

If instead $c$ is idempotent, then
$ce,ec\in E(S)=\{c,e\}$. We show that $cb\in Q(S)$.
If $cb=a$, idempotence of $c$ gives $ca=a$, whence
$cb=ca^2=(ca)a=b$, a contradiction.
If $cb=d$, put $v=ca$. Then $va=d$.
The possibilities $v=a,b,e$ give $va=b,e,e$, respectively.
The possibility $v=c$ would give $cb=c$.
Thus $v=d$ and $da=d$. It follows that
\[
 d=da^3=de=cbe=ce\in E(S),
\]
contrary to $d\notin Q(S)$. The proof for $bc$ is dual.

Finally, suppose that both $c$ and $d$ are squares. If either
is nonidempotent, its root must be the other, and these two
elements together with $e$ form a group of order three.
Indeed, a cyclic $N_3$ would make its generator nonsquare by
the argument already used for $a$.
Since $be=eb=e$, multiplication of $b$ by an element of this
group gives that group element. Closure follows.
If $c,d$ are both idempotent, closure of $E(S)$ handles all
products except those involving $b$. For an idempotent $h$,
$bh=a$ would imply $ah=a$ and hence
$bh=a^2h=a(ah)=b$. Similarly $hb\ne a$.
As $a$ is the only nonsquare, this completes the proof.
\end{proof}

\begin{lemma}\label{lem:nonsquare-products}
Let $S$ be an ai-semiring of order at most five. If
$b\in K(S)\setminus Q(S)$, then
\begin{equation}\label{eq:right-stable}
 bz=bz^2\qquad\text{for every }z\in S.
\end{equation}
\end{lemma}
\begin{proof}
Since $Q(S)$ is not a subsemigroup, Lemma~\ref{lem:orthodox-small}
provides idempotents $e,f$ for which $v=ef$ is not idempotent.
The three elements $e,f,v$ are distinct. Put
\[
 w=fe,\qquad r=efe,\qquad s=fef.
\]
None of $w,r,s$ belongs to $\{e,f,v\}$. For $w$ this follows
because $fe=e,f$, or $ef$ would imply $(ef)^2=ef$.
For $r$, equality to $e$ or $v$ gives the same consequence,
whereas $r=f$ and $er=r$ would give $ef=f$.
The argument for $s$ is dual. There are at most two available
elements, so at least one of $w=r$, $w=s$, or $r=s$ holds.

Suppose first that $w=r$. Associativity gives
\[
 w^2=w,\quad s=wf,\quad s^2=s,\quad v^2=v^3=s.
\]
If $s\ne w$, these five distinct elements exhaust $S$, all
except $v$ are idempotent, and the displayed equations prove
\eqref{eq:right-stable} for its only nonsquare $v$.
The case $w=s\ne r$ is dual: again all elements except $v$
are idempotent and $v^2=v^3=r$, which is enough for
\eqref{eq:right-stable} on the right as well.

If $r=s=h\ne w$, the five elements are $e,f,v,w,h$.
Associativity gives
\[
 h^2=h,\qquad v^2=w^2=vw=wv=h,
\]
and $h$ absorbs multiplication by each of these five elements.
The only possible nonsquares are $v,w$. Multiplication by
an idempotent satisfies~\eqref{eq:right-stable} automatically;
multiplication by $v$ or $w$ gives $h$. This proves the claim.

The remaining case has $w=r=s=o$. Then
$T=\{e,f,v,o\}$ is a subsemigroup with table
\[
\begin{array}{c|cccc}
 &e&f&v&o\\\hline
e&e&v&v&o\\
f&o&f&o&o\\
v&o&v&o&o\\
o&o&o&o&o
\end{array}
\]
and its only nonsquare is $v$. If there is an extra element
$z$, any nonsquare in $K(S)$ is still $v$: if $z$ is not a
square then all squares, and hence their products, lie in $T$;
if $z$ is a square it cannot be the element under consideration.
It remains to prove $vz=vz^2$.

The values $vz=v,o$ give the result immediately; for the latter,
$oz=v^2z=v(vz)=o$. The value $vz=f$ is impossible because
$e(vz)=vz$ but $ef=v\ne f$.
If $vz=e$, the equations
\[
 v(zf)=v,\qquad v(ze)=e
\]
and the displayed table force $zf=f$ and $ze=z$. Hence
$zv=zef=f$, and $(ez)v=e(zv)=v$ forces $ez=e$.
Therefore $vz^2=ez=e=vz$.

Finally suppose that $vz=z$. Then $oz=z$, and the fixed
elements of left multiplication by $v$ are exactly $o,z$.
Since $v z^2=z^2$, we have $z^2\in\{o,z\}$.
If $z^2=o$, the equations $oz=z$ and $z^2=o$ show that
$o,z$ are incomparable: multiplying either proposed comparison
on the right by $z$ gives its reverse.
But $o+z$ must also be fixed by left multiplication by $v$,
so it would equal $o$ or $z$, a contradiction.
Thus $z^2=z$, completing the proof.
\end{proof}

\subsection{A square and one right multiplier}

We record three elementary facts used below. Distinct elements
of a finite multiplicative subgroup of an ai-semiring are
incomparable: a comparison can be translated to $1\leq g$,
and iteration around the finite order of $g$ forces equality.
If such a subgroup is $C_3$ and $|S|\leq5$, the join of any
two distinct group elements contains the third. To see this,
the join lies outside the group; its orbit under the group has
at most two elements, but its size divides three. The join is
therefore fixed, and translation supplies the third element.
Similarly, a subgroup $C_4$ in an ai-semiring of order at most
five forces the entire semiring to be its flat extension $G$:
every join of distinct group elements is the unique remaining
element, which is absorbing for multiplication by distributivity.

We also use the elementary finite-semigroup implication
\begin{equation}\label{eq:group-element}
 q=q^2z\quad\Longrightarrow\quad
 q\text{ belongs to a cyclic subgroup}.
\end{equation}
Indeed, $q^jz=q^{j-1}$ for $j\geq2$. If
$q^r=q^{r+d}$, multiply by $z^{r-1}$ to obtain
$q=q^{d+1}$, so the cyclic subsemigroup of $q$ is a group.

Define two five-element ai-semirings $E_1,E_2$ on
$\{z,b,q,f,t\}$. In both, addition is determined by
\begin{equation}\label{eq:exception-order}
 z<q<t,\qquad z<f<t,\qquad b<f<t,
 \qquad b+q=t,
\end{equation}
with no further comparabilities. For $E_1$, the elements
$z,b,q$ form the cyclic semigroup
$z^2=b$, $z^3=q$, $z^4=b$, and every product involving $f$
or $t$ is $t$. For $E_2$, multiplication is cyclic:
\[
 z^2=b,\qquad z^3=q,\qquad z^4=f,\qquad z^5=z^6=t.
\]
The ensuing proof only needs the fact that these descriptions
are forced whenever they arise. Their semiring laws can also
be checked directly from the displayed descriptions.

\begin{lemma}\label{lem:three-exceptions}
Let $S$ be an ai-semiring of order at most five, and let $b$
be a square. If, for some $z\in S$,
\begin{equation}\label{eq:two-variable-failure}
 bz\not\leq b+bz^2+z,
\end{equation}
then $S$ is $E_1$, $E_2$, or $G$, with the following failing
elements: the displayed $b,z$ in $E_1,E_2$, or
$b=z=g^2$ in $G$. In particular, except in $G$,
\begin{equation}\label{eq:local-square}
 ubz\leq ub+ubz^2+uz\qquad(u,z\in S,\ b\in Q(S)).
\end{equation}
\end{lemma}
\begin{proof}
Fix a failure, and put
\[
 q=bz,\quad s=bz^2,\quad f=b+s+z,\quad t=f+q.
\]
If $b\leq q$, right multiplication by $z$ gives $q\leq s$,
a contradiction. Thus $b\not\leq q$. Since $b\leq f$ and
$q\not\leq f$, the elements $q,f$ are incomparable and
$q,f,t$ are distinct. Also $b,q$ are incomparable.

\smallskip\noindent
\emph{Idempotent $b$.}
Suppose $b^2=b$. Neither $f=b$ nor $z=b$ is possible.
If $f=z$, left multiplication by $b$ gives
$q=bf=b+s+q$, contradicting $b\not\leq q$.
Thus $b,z,f,q,t$ are five distinct elements and exhaust $S$.
Necessarily $b+q=t$. Left multiplication by $b$ fixes $b,q,t$,
sends $z$ to $q$, and sends $f$ to
$b+s+q=t$. It fixes $s$, and $s\leq f$, so $s=b$.
The equation $bz^2=b$ then forces $z^2=b$, because $b$ is
the only preimage of $b$ under that left multiplication.
Hence $z^3=q$, $z^4=b$, and $f=b+z$.
The join $z+q$ is $q$ or $t$. The latter possibility would
give $(z+q)z=b$ while $tz=(b+q)z=q+b=t$.
Thus $z<q$, and all of~\eqref{eq:exception-order} follows.
Distributing the products of $f=b+z$ and $t=b+q$ now gives
exactly $E_1$.

For the remaining cases suppose that $b=a^2$ is not idempotent.

\smallskip\noindent
\emph{Case 1: $b,z,f$ are distinct.}
They exhaust $S$ together with $q,t$, and $b+q=t$.
The root $a$ cannot be $f$ or $t$: either choice dominates
both $b,z$, giving $q=bz\leq a^2=b$.
Nor is $a=b$. If $a=q$, then $q=q^2z$, so $q$ lies in a
cyclic group $H$ and $b=q^2\in H$.
Both $f,t$ are outside $H$, since they strictly dominate $b$.
If $z\notin H$, the size bound gives $|H|\leq2$, making
$b$ idempotent. If $z\in H$, then $z=q^{-1}$; group orders
two and three would give $z=q$ and $z=b$, respectively,
and larger orders require at least six elements including
$f,t$. All possibilities are excluded.
Consequently $a=z$, so $b=z^2$, $q=z^3$, and
$s=z^4\in\{z,b,f\}$.
The value $s=z$ gives a group $C_3$, where the join of $z,b$
contains $q$. The value $s=b$ makes $b$ idempotent.
Thus $s=f$.
Since $z,b\leq f$, multiplication on the right by $z$ gives
$z^5=fz\geq b,q$, whence $z^5=t$ and $z^6=t$.
Neither $z\leq b$ nor $b\leq z$ is possible, since right
multiplication by $z$ would give $b\leq q$ or $q\leq b$.
Thus $z+b=f$. If $z+q=t$, multiplication by $z$ would give
$t=tz=b+f=f$. Therefore $z+q=q$, yielding exactly $E_2$.

\smallskip\noindent
\emph{Case 2: $f=z\ne b$.}
Here $b<z$, and $a\notin\{b,z,t\}$.
If $a=q$, then~\eqref{eq:group-element} applies.
The comparison $b\leq z$ gives $q^4=b^2\leq bz=q$.
Comparability within the cyclic group forces $q^4=q$.
Its order is three, and $s=qz=bq$ is its identity.
The $C_3$ join property yields $q\leq b+s\leq f$, a
contradiction.
Otherwise $S=\{a,b,z,q,t\}$ with five distinct elements.
Since $a^4=b^2\leq q$, its only possible values are $a,q$.
If $a^4=a$, the group $\{a,a^2,a^3\}\cong C_3$ has no
place for its third element: $z,t$ dominate $b$, while
$q$ dominates $a$.
If $a^4=q$, the value $a^3$ cannot be $a$ or $b$, and
cannot be $z$ or $t$ either: these dominate $b$, which would
give $q=a^4\geq a^3$. Thus $a^3=q$ and all later powers
equal $q$. But then $s=qz=b^2z=bq=q$, again a contradiction.

\smallskip\noindent
\emph{Case 3: $z=b\ne f$.}
We have $q=b^2$ and $f=b+b^3$.
The root $a$ cannot dominate $b$, because this would give
$b=a^2\geq b^2=q$. Thus $a\notin\{b,f,t\}$.
If $a=q$, then $b^4=b$, yielding a group $C_3$ whose join
$b+b^3$ contains $q=b^2$, a contradiction.
Otherwise $S=\{a,b,q,f,t\}$. None of these five elements
can be $a^3$: the values $a,b$ contradict $a^4=q$; the
value $q$ would make all later powers equal $q$ and hence
give $b^3=q\leq f$; and the values $f,t$ dominate $b$,
forcing $q=a^4\geq a^3\geq b$.

\smallskip\noindent
\emph{Case 4: $f=b$.}
Here $z,s\leq b$ and $q\leq b^2$.
The root $a$ is not $b$ or $t$, and cannot be $z$,
since $z\leq z^2$ would give $b=z^2\leq z^3=q$.
If $a=q$, then $q=q^2z$ gives a cyclic group, and
$q\leq b^2=q^4$ forces $q=q^4$. Its order is three;
then $s=qz=bq$ is its identity, contradicting $s\leq b$.
Thus $a\notin\{b,q,t,z\}$.

First suppose $b^2=q$. We obtain
$a^4=q$ and $a^6=b^3=qz=s\leq b$.
The element $a^3$ differs from $a,b,q,t$: the first two
values contradict $a^4=q$, the third would give $a^6=q$,
and the fourth would imply $q=a^4\geq a^3=t$.
Hence $S=\{a,b,a^3,q,t\}$.
The value $a^5=b$ would give $a^3=a^6\leq b$ and then
$q\leq a^3\leq b$.
The values $a^5=a^3,q$ give $a^6=q$.
The value $a^5=t$ gives
$a^6=at=a^3+a^5\geq q$.
All contradict $q\not\leq b$. Therefore $a^5=a$,
and the four powers of $a$ form $C_4$. Thus $S=G$;
its flat order forces $z=b=a^2$.

Finally suppose $b^2\ne q$. Then $z\ne b$, so
$S=\{a,z,b,q,t\}$. As $b^2\geq q$, its only possible
values are $a,t$. If $b^2=a$, then $a^4=a$ gives a group
$C_3$, whose third element cannot be $z,t,q$: they satisfy
$z\leq b$, $b\leq t$, and $q\leq a$, respectively.
Thus $b^2=t$.
The only possible values of $a^3$ are $q,t$:
$a,b$ would give $a^4=b$, and $z\leq b$ would give
$t=a^4\leq a^3=z\leq b$.
Put $v=az$, so $av=q$.
If $a^3=q$, then $v\leq q$, leaving only $v=a,q,z$;
these give $av=b,t,z$, respectively, not $q$.
If $a^3=t$, then $at=t$.
The equation $av=q$ excludes $v=a,b,t$; the value $v=z$
would also give $av=z$. Consequently $v=q$ and $aq=q$.
Now $as=a(qz)=s$, but the only fixed elements of left
multiplication by $a$ are $q,t$, neither of which lies
below $b$. This contradicts $s\leq b$ and finishes the
classification.

To obtain~\eqref{eq:local-square}, multiply the two-variable
inequality on the left by $u$ whenever it holds.
In its failing instances in $E_1,E_2$, both $uf$ and $ut$
equal $t$ for every $u$. Since $q\leq t$, left
multiplication still gives $uq\leq uf$.
Only $G$ remains exceptional.
\end{proof}

\subsection{Products of arbitrarily many squares}

\begin{lemma}\label{lem:guard-theoretical}
Let $S$ be an ai-semiring of order at most five, not isomorphic
to $G$. For every $k\geq2$, all squares $a_1,\ldots,a_k$,
and every $z\in S$, put $p=a_1\cdots a_k$. Then
\begin{equation}\label{eq:guard}
 pz\leq p+pz^2+a_1z.
\end{equation}
\end{lemma}
\begin{proof}
Put $b=a_2\cdots a_k\in K(S)$. If $b$ is a square,
apply~\eqref{eq:local-square} with $u=a_1$.
Otherwise Lemma~\ref{lem:nonsquare-products} gives $bz=bz^2$,
so $pz=pz^2$. Both cases prove~\eqref{eq:guard}.
\end{proof}

\begin{proposition}\label{prop:five-validity}
The variety $\Avar_5$ satisfies $\sigma_k$ for every
$k\in\mathcal N$.
\end{proposition}
\begin{proof}
Let $S$ be a five-element ai-semiring and fix a valuation.
If $S\not\cong G$, apply Lemma~\ref{lem:guard-theoretical} with
$a_i=x_i^2$, and add the summands
$y_i^2y_{i+1}^2x_i^2z$ to its right side.
The resulting inequality is $q_k\leq\poly f_k$, with the
multiplication order displayed in~\eqref{eq:noncomm-poly}.
If $S\cong G$, use Lemma~\ref{lem:test-validity}.
Thus every generator, and hence their generated variety,
satisfies $\sigma_k$.
\end{proof}

\begin{theorem}\label{thm:five}
The variety $\Avar_5$ has no equational basis with a uniform
bound on the number of variables. In particular, it is
nonfinitely based.
\end{theorem}
\begin{proof}
Both $H^0$ and $G$ are five-element ai-semirings.
Apply Theorem~\ref{thm:criterion} and
Proposition~\ref{prop:five-validity}.
\end{proof}

\Needspace{7\baselineskip}
\begin{corollary}\label{cor:interval}
Every ai-semiring variety $\mathcal W$ with
$\TT\subseteq\mathcal W\subseteq\Avar_5$ has no basis with a
uniform bound on the number of variables.
\end{corollary}
\begin{proof}
The lower inclusion supplies the two auxiliary semirings and the
upper inclusion supplies all the identities $\sigma_k$.
Apply Theorem~\ref{thm:criterion}.
\end{proof}

\section{Orders at least six}
\label{sec:six}

A finite algebra is \emph{inherently nonfinitely based} if it
belongs to no locally finite finitely based variety of its
signature. For ai-semirings, restricting the definition to
varieties of ai-semirings gives the same condition. Indeed,
intersecting a locally finite finitely based variety with the
finitely axiomatizable variety $\AI$ again gives a locally
finite finitely based variety.

Define the Zimin words recursively by
$Z_1=x_1$ and $Z_{r+1}=Z_rx_{r+1}Z_r$. For an ai-semiring
$S$, a word $w$ is \emph{maximal} if $S\models w\absorb u$
for a word $u$ forces $u=w$ as literal words. It is
\emph{minimal} if $S\models u\absorb w$ forces $u=w$.
A word having both properties is \emph{isolated}.

We use the following three results from
Dolinka, Gusev, and Volkov~\cite{DGV25}:
\begin{enumerate}
\item an ai-semiring generating a locally finite variety is
inherently nonfinitely based if all its Zimin words are
isolated \cite[Proposition~1]{DGV25};
\item all Zimin words are maximal in the six-element
ai-semiring $\overline{\mathcal A_2^1}$
\cite[Lemma~1]{DGV25};
\item all Zimin words are minimal in the six-element
ai-semiring $\mathcal B_2^1$ \cite[Lemma~2]{DGV25}.
\end{enumerate}
We use the notation of that paper for the last two algebras;
in particular, the overline in
$\overline{\mathcal A_2^1}$ specifies its addition.

\begin{proposition}\label{prop:six}
For every $n\geq6$, the variety $\Avar_n$ contains an
inherently nonfinitely based $36$-element ai-semiring.
Consequently $\Avar_n$ and $\VV_n$ are nonfinitely based;
neither is contained in a locally finite finitely based
variety of the same signature.
\end{proposition}
\begin{proof}
Put
\[
M=\overline{\mathcal A_2^1}\times\mathcal B_2^1.
\]
The two factors belong to $\Avar_6$, so $M\in\Avar_6$.
If $M\models Z_r\absorb u$, projection onto the first
factor and maximality force $u=Z_r$. If
$M\models u\absorb Z_r$, projection onto the second factor
and minimality give the same conclusion. Every Zimin word
is therefore isolated in $M$. As $M$ is finite, it generates
a locally finite variety; the cited criterion makes $M$
inherently nonfinitely based.

By Lemma~\ref{lem:padding}, both six-element factors belong
to every $\Avar_n$ with $n\geq6$, and hence so does $M$.
The inclusion $\Avar_n\subseteq\VV_n$ places $M$ in $\VV_n$
as well. Both varieties are locally finite. A locally finite
finitely based variety containing either one would contain
$M$, contrary to its inherent nonfinite basability.
\end{proof}

\begin{remark}
The proof of Proposition~\ref{prop:six} also applies if the
larger joint variety is defined using semirings whose addition
need not be commutative. The witness already has commutative
idempotent addition. The passage in the next section, in
contrast, uses additive commutativity.
\end{remark}

\section{Additive retractions and the classification}
\label{sec:transfer}

For a positive integer $m$, write $ms$ for the sum of $m$
copies of an element $s$, and use the same convention for
terms.

\begin{proposition}\label{prop:intersection}
For every positive integer $n$,
\begin{equation}\label{eq:intersection}
\VV_n\cap\AI=\Avar_n.
\end{equation}
\end{proposition}
\begin{proof}
Let $N=n!$, and let $S$ be an $n$-element semiring with
commutative addition. Every cyclic subsemigroup of $(S,+)$
has index and period at most $n$. Since $N\geq n$ and $N$
is divisible by every possible period,
\begin{equation}\label{eq:multiple}
2Ns=Ns=N^2s\qquad(s\in S).
\end{equation}
For $n=1$ these equalities hold trivially.

The map $e:S\to S$ given by $e(s)=Ns$ satisfies
\[
e(a+b)=e(a)+e(b),
\qquad e(a)e(b)=N^2(ab)=N(ab)=e(ab).
\]
Here additive commutativity and distributivity justify the
expansion of the products and sums, and~\eqref{eq:multiple}
gives the middle equality. Thus $e$ is an idempotent
endomorphism whose image is an ai-semiring of order at most
$n$. By Lemma~\ref{lem:padding}, $e(S)\in\Avar_n$.

If an identity $u\approx v$ holds in $\Avar_n$, then for
every valuation $\theta$ in $S$, the homomorphism property
of $e$ gives
\[
N u(\theta)=e(u(\theta))
=u(e\circ\theta)=v(e\circ\theta)
=e(v(\theta))=N v(\theta).
\]
Hence $Nu\approx Nv$ holds in every generator of $\VV_n$,
and therefore in $\VV_n$. In an ai-semiring this identity
reduces to $u\approx v$. Thus every member of
$\VV_n\cap\AI$ satisfies all identities of $\Avar_n$, proving
one inclusion in~\eqref{eq:intersection}. The other is immediate.
\end{proof}

\begin{corollary}\label{cor:transfer}
If $\Avar_n$ is nonfinitely based, then $\VV_n$ is nonfinitely
based. In particular, $\VV_5$ is nonfinitely based.
\end{corollary}
\begin{proof}
A finite basis for $\VV_n$, together with $x+x\approx x$,
would be a finite basis for $\VV_n\cap\AI=\Avar_n$.
Apply Theorem~\ref{thm:five} when $n=5$.
\end{proof}

\begin{proof}[Proof of Theorem~\ref{thm:main}]
For $n=1$ both varieties are trivial and are defined by
$x\approx y$. For $n=2$ use~\cite{SR15,WWLY26}.
The varieties $\Avar_3$ and $\Avar_4$ are nonfinitely based
by~\cite{WJW26,YS26}; $\Avar_5$ is nonfinitely based by
Theorem~\ref{thm:five}; and all $\Avar_n$ with $n\geq6$ are
nonfinitely based by Proposition~\ref{prop:six}.
Corollary~\ref{cor:transfer} gives the corresponding conclusion
for every $\VV_n$ with $n\geq3$.
\end{proof}

\begin{remark}
Equation~\eqref{eq:intersection} concerns the constant-free
signature and the convention of commutative addition.
It does not assert the same statement for a signature with
distinguished constants $0$ or $1$, or for arbitrary
noncommutative addition. No such extension is needed for
Theorem~\ref{thm:main}.
\end{remark}

\end{document}